\documentclass{article}

\usepackage[utf8]{inputenc}
\usepackage{amsmath}
\usepackage{amsfonts}
\usepackage{amssymb}
\usepackage{amsthm}
\usepackage{amscd}
\usepackage{bbding}
\usepackage{tikz-cd}
\usepackage[english]{babel}
\usepackage[square,sort,comma,numbers]{natbib}

\theoremstyle{plain}
\newtheorem{theorem}{Theorem}
\newtheorem{lemma}[theorem]{Lemma}

\newtheorem{proposition}[theorem]{Proposition}
\newtheorem{definition}[theorem]{Definition}

\theoremstyle{remark}
\newtheorem*{corollary}{Corollary}
\newtheorem*{remark}{Remark}
\newtheorem*{example}{Example}

\title{Pushouts of affine algebraic sets}

\author{Jakub Kop\v{r}iva}

\date{January 11, 2021}

\begin{document}

\maketitle

\begin{abstract}
	We study the problem of existence of pushouts in the category of algebraic sets over an infinite field. This problem can be reduced to asking whether the property of being a finitely generated algebra over a field, or a Noetherian ring in general, is preserved under taking pullbacks. We show that this problem can be fully solved in terms of preservation of this property under taking intersections. In addition, we also discuss specific cases of pushouts of algebraic sets and intersections of Noetherian rings.
\end{abstract}

\noindent \textbf{Key words:} Pullbacks, pushouts, Noetherian rings, finitely generated algebras\\
\noindent \textbf{MSC2010:} 13E05, 13E15 (Primary) 13B30 (Secondary)

\section*{Introduction}
In this article, we study the problem of existence of pushouts in the category of algebraic sets over an infinite field $K$. This seemingly natural question is often sidestepped by working in a larger category (such as schemes in \cite{Fer} and \cite{Sch}) since affine algebraic sets are known not to be closed as a category under similar constructions. However, we find this question interesting from the point of view of commutative algebra.

We prove that there exists a pushout of a diagram of $K-$algebraic sets if and only if the pullback of the corresponding diagram of coordinate rings is a finitely generated $K-$algebra. This leads us to study the question of whether a pullback of a diagram of commutative Noetherian rings or finitely generated algebras over a commutative Noetherian ring $R$ is Noetherian or a finitely generated $R-$algebra, respectively.

Using the primary decomposition of the zero ideal, we show that, effectively, this question is equivalent to testing whether some extensions of rings or $R-$algeras are finite (meaning that the larger ring or $R-$algebra is a finitely generated module over the smaller one) and whether some interesections of two Noetherian rings or finitely generated algebras over $R$ are Noetherian or finitely generated over $R$, respectively. Thereby, we obtain generalizations of related results of \cite{Fon} and \cite{BW}.

Finally, we treat some examples. We prove that in specific cases an interesection of two Noetherian rings or finitely generated algebras over $R$ is, indeed, Noetherian or finitely generated over $R$, respectively. In one case, we actually make use of the very tools we develop to tackle the preservation of being Noetherian (a finitely generated algebra over a commutative Noetherian ring) under forming pullbacks. We also study some pushouts of algebraic sets with focus on their local properties.

\section{Pushouts of affine algebraic sets}\label{Sec1}
	Fix $K$ an infinite field, and consider, for example, these diagrams of algebraic sets (over $K$) with natural maps:

\begin{center}
	\begin{tikzcd}
 		& \mathbb{A}_K^1 \arrow{dl} \arrow{dr}{(x,0)} &\\
 		\mathbb{A}_K^0 & &\mathbb{A}_K^2
	\end{tikzcd}
	\begin{tikzcd}
 		& \mathbb{A}_K^0 \sqcup \mathbb{A}_K^0 \arrow{dl} \arrow{dr}{\{0\} \sqcup \{1\}} &\\
 		\mathbb{A}_K^0 & &\mathbb{A}_K^1
	\end{tikzcd}
\end{center}

The pushout of the diagram on the left, $P$, would correspond to contracting the $x-$axis (or a line in general) in a two dimensional affine space into a single point. Similarly, the pushout on the right, $Q$, would correspond to contracting two points on an affine line into a single point.

It is folklore knowledge that the first pushout does not exists in the category of affine algebraic sets (see Example 3.5 on page 7 in \cite{Sch}), but the second one can be constructed; it is isomorphic to the curve $V(y^2 - yx - x^3) \subseteq \mathbb{A}^2_k$. We discuss this example and generalizations of the pushout $Q$ thereof below in Section \ref{Sec5}.

In the following text, we try to characterise under which circumstances it is possible to form a pushout of a diagram of affine algebraic sets with polynomial maps (naturally, in the category of algebraic sets). In other words, we want to know when we can glue two affine algebraic sets via another algebraic set that is mapped to both of them.

The existence of pushouts $P$ and $Q$, respectively, would, through the contravariant equivalence between the category of algebraic sets and the category of coordinate rings (for general facts of classical algebraic geometry, we refer the reader to \cite{Eis}), mean that the coordinate rings $K[P]$ and $K[Q]$ are pullbacks in the category of coordinate rings of the respective diagrams.

Moreover, we show that if $K[P]$ and $K[Q]$ were finitely generated as algebras over $K$, they would need to be the pullbacks of the respective diagrams in the category of $K-$algebras. This is due to the following lemmas:

\begin{lemma}\label{L1}
	Let $X$ be a non-empty $K-$algebraic set. Then, any finitely generated $K-$subalgebra of $K[X]$ is a coordinate ring of some algebraic set over $k$.
\end{lemma}

\begin{proof}
	Let $B$ be a finitely generated $K-$subalgebra of $K[X]$ with generators $\Phi_1, \dots, \Phi_n.$ There exists a unique homomorphism $\varphi: K[x_1, \dots, x_n] \to B$ such that $x_i \mapsto \Phi_i$ for all $i \in \{1, \dots, n\}$. However, $\varphi$ can be thought of as a homomorphism from $K[x_1, \dots, x_n]$ to $K[X]$. Then, due to $K$ being an infinite field, $\varphi$ gives rise to a polynomial map $\Phi = (\Phi_1, \dots, \Phi_n): X \to \mathbb{A}^n_K.$ Naturally, this map factorizes through the inclusion $\overline{\Phi(X)} \subseteq \mathbb{A}_k^n$. This means that $\varphi$ factorizes through the projection $K[x_1, \dots, x_n]/I(\Phi(X))$; therefore, $\mathrm{Ker}\;\varphi$ contains $I(\Phi(X)).$ On the other hand, let us have $f \in \mathrm{Ker}\,\varphi$. Then: $$0 = f(\Phi_1, \dots, \Phi_n)(x) = f(\Phi_1(x), \dots, \Phi_n(x))$$ for all $x \in X$, hence $f$ vanishes on $\Phi(X)$, and we established that $\mathrm{Ker}\,\varphi = I(\Phi(X))$, which means that $B \cong K[x_1, \dots, x_n]/I(\Phi(X))$.
\end{proof}

\begin{lemma}\label{L2}
	Suppose $\varphi: B \to A$ a $\psi: C \to A$ are homomorphisms of coordinate rings of affine algebraic sets over $K$. Then, $D$, with naturally denoted morphisms $\theta_B, \theta_C$, is the pullback of the corresponding diagram in the category of coordinate rings if and only if $D$ is also the pullback of corresponding diagram in the category of all $K-$algebras.
\end{lemma}

\begin{proof}
	We know that the pullback of the corresponding diagram in the category of all $K-$algebras is of the form $P = \{(b,c) \in B \times C\,|\, \varphi(b)=\psi(c)\}$ with naturally defined $\eta_B, \eta_C$ projections. However, $B$ and $C$ are coordinate rings of $K-$algebraic varieties $X$ and $Y$, respectively. This implies that $B \times C$ is a coordinate ring of $X \sqcup Y.$ Therefore, if $P$ is finitely generated subalgebra thereof, then is a coordinate ring of an algebraic set over $K$ by Lemma \ref{L1}.
	
	$(\Leftarrow)$ Provided that $P$ is finitely generated as an algebra over $K$, it is also a coordinate ring, and $D \cong P$ by the universal property of pullback.
	
	$(\Rightarrow)$ Suppose that $D$ is not pullback of corresponding diagram in the category of all $K-$algebras. Then, $P$, by the previous paragraph, cannot be finitely generated as an algebra over $K$. By virtue of $P$ being pullback of the diagram in the category of all $K-$algebras, there is a homomorphism $\varrho: D \to P$ such that projections from $D$ to $B$ and $C$ factor through it. Therefore, $\theta_B = \eta_B\rho$ and $\theta_C = \eta_C\rho$. Clearly, this means that $\varrho(D)$ with projections defined as restrictions of $\eta_B$ and $\eta_C$ is also the pullback in the category of finitely generated $K-$algebras. This is due to the universal property of the pullbacks $D$ and $P$ in their respective categories.
	
	As $P$ is not finitely generated, either $\mathrm{Im}\,\eta_B$ or $\mathrm{Im}\,\eta_C$ is not finitely generated (this follows from Proposition \ref{P6} below). Without loss of generality, assume that $\mathrm{Im}\,\eta_B$ is not finitely generated, and find a finitely generated subalgebra $D'$ of $P$ such that $\eta_B(\varrho(D))$ is strictly smaller than $\eta_B(D').$ We deduce by Lemma \ref{L1} that $D'$ is a coordinate ring.
	
	This yields a contradiction as restriction of $\eta_B$ to $D'$ clearly does not factor through $\eta_B$ restricted to $\varrho(D)$ due to $\eta_B(\varrho(D))$ being strictly smaller than $\eta_B(D')$, the assumed pullback of the diagram in the category of coordinate rings.
\end{proof}
	
Therefore, we have successfully translated a geometric question about existence of pushouts of $K-$algebraic sets to the question of whether ring-theoretic pullbacks of induced diagrams of their coordinate rings are finitely generated $K-$algebras.

Below, we study this problem in a general setting of pullbacks of diagrams of communative unital Noetherian rings and finitely generated algebras over them. There are already some results in speacial cases when one of the ring homomorphims is surjective by \cite{Fon} and \cite{BW}.
 
\section{Preliminaries}\label{Sec2}
We begin with some preliminary results and discussion on (pullbacks of) communative unital Noetherian rings and finitely generated algebras over them.

\begin{theorem}[Hilbert basis theorem; Theorem 1.2 and Corollary 1.3 in \cite{Eis}, pages 27 and 28]\label{T3}
	If a commutative unital ring $R$ is Noetherian, then the polynomial ring $R[x]$ is Noetherian. Furthermore, any finitely generated algebra over $R$ is Noetherian.
\end{theorem}

Henceforth, $R$ will denote a Noetherian commutative unital ring. All other rings are commutative unital, unless otherwise indicated.

\begin{theorem}[Eakin-Nagatan theorem in \cite{Nag}; Artin-Tate lemma in \cite{Eis}, Theorem in Exercise 4.32, page 14
]\label{T4}
	Suppose $T \subseteq S$ are rings ($R$-algebras). If $S$ is Noetherian (a finitely generated $R$-algebra) and a finitely generated $T-$module, then $T$ is Noetherian (a finitely generated $R$-algebra). 
\end{theorem}

\begin{proposition}\label{P5}
	Suppose $S$ is a ring (an $R-$algebra) with ideals $I, J \subseteq S$ such that both $S/I$ and $S/J$ are Noetherian (finitely generated algebras over $R$). Then, $S/I \cap J$ is Noetherian (a finitely generated $R-$algebra).
\end{proposition}
\begin{proof} For $S/I$ and $S/J$ Noetherian, the proposition clearly follows from the fact that Noetherian modules are closed under submodules, factors, and extensions.

	Suppose, therefore, that $S/I$ and $S/J$ are Noetherian finitely generated algebras over $R$. Without loss of generality, we can assume that $I \cap J = \{0\}.$ Otherwise, we set $\tilde{S}=S/I \cap J$, $\tilde{I} = I/I \cap J$, and $\tilde{J} = J/I \cap J.$
	
	Let us denote the canonical projections by $\pi_I: S \to S/I$ and $\pi_J: S \to S/J$. If we lift finitely many generators of $S/I$ and $S/J$, we obtain $S_I$ and $S_J$ finitely generated $R-$subalgebras of $S$ such that $\pi_I(S_I)=S/I$ and $\pi_J(S_J)=S/J$, respectively.
	
	Additionally, we have that $I + J/J \cong I/I \cap J \cong I$ as $I \cap J = 0.$ Since $S/J$ is Noetherian, by the Hilbert basis theorem, it can be viewed as a Noetherian $S_J-$module, and $I$ can be viewed as isomorphic to its submodule. We conclude that $I$ is a Noetherian, hence finitely generated $S_J-$module.
	
	Choose an arbitrary $s \in S$. Since $\pi_I(S_I) = S/I$, there exists an $s_I \in S_I$ such that $s - s_I \in I$. However, any element $i \in I$ can be expressed as $i = \sum_{k=1}^n \iota_k s_{k,J}$ for fixed $\iota_1, \dots, \iota_k \in I$ and some $s_{1, J}, \dots, s_{k, J} \in S_J.$	Consequently, $S$ is generated as a $R-$algebra by finitely many generators of $S_I$ and $S_J$ together with finitely many generators of $I$ as a $S_J-$module. 
\end{proof}

\begin{proposition}\label{P6}
	Let $A,B,C$ be Noetherian (finitely generated $R-$algebras), and let $\varphi: B \to A$ a $\psi: C \to A$ be their homomorphisms. Then, the pullback of the corresponding diagram: $$
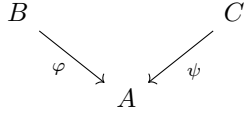
\begin{tikzcd}
			B \arrow{dr}[']{\varphi} & & C \arrow{dl}{\psi}\\
			& A &	
		\end{tikzcd}$$ is Noetherian (a finitely generated $R-$algebra) if and only if both $\varphi^{-1}(\mathrm{Im}\,\varphi \cap \mathrm{Im}\,\psi)$ and $\psi^{-1}(\mathrm{Im}\,\varphi \cap \mathrm{Im}\,\psi)$ are Noetherian (finitely generated as algebras over $R$).
\end{proposition}
\begin{proof}
	Pullback of the diagram above exists in the category of commutative rings, and it can be expressed as $P=\{(y,z) \in B \times C, \varphi(y) = \psi(z)\}$ with projections $\pi_1: P \to B, (y,z) \to y$ and $\pi_2: P \to C, (y,z) \to z$. The ring $P$ can be naturally equipped with $R-$algebra structure such that $\pi_1$ and $\pi_2$ become $R-$algebra homomorphisms. It is clear that $\pi_1(P) = \varphi^{-1}(\mathrm{Im}\,\varphi \cap \mathrm{Im}\,\psi)$ and $\pi_2(P) = \varphi^{-1}(\mathrm{Im}\,\varphi \cap \mathrm{Im}\,\psi)$. Also, we have that $\mathrm{Ker}\,\pi_1 = (0, \mathrm{Ker}\,\psi)$ and $\mathrm{Ker}\,\pi_2 = (\mathrm{Ker}\,\varphi, 0)$.
	
	$(\Rightarrow)$ We observe that $\mathrm{Ker}\,\pi_1 \cap \mathrm{Ker}\,\pi_2 = \{(0,0)\}$. Under the assumption that $\pi_1(P) = \varphi^{-1}(\mathrm{Im}\,\varphi \cap \mathrm{Im}\,\psi)$ and $\pi_2(P) = \varphi^{-1}(\mathrm{Im}\,\varphi \cap \mathrm{Im}\,\psi)$ are Noetherian (finitely generated $R-$algebras), $P$ is Noetherian (a finitely generated as an algebra over $R$) from Proposition \ref{P5}.
	
	$(\Leftarrow)$ If $P$ is Noetherian (a finitely generated $R-$algebra), then both rings ($R$-algebras) $\varphi^{-1}(\mathrm{Im}\,\varphi \cap \mathrm{Im}\,\psi)$ and $\psi^{-1}(\mathrm{Im}\,\varphi \cap \mathrm{Im}\,\psi)$ are Noetherian (finitely generated over $R$) as its homomorphic images under $\pi_1$ and $\pi_2$, respectively.
\end{proof}
	
If both morphims are injective, then the taking the pullback means simply taking the intersection of images, and we are none the wiser. Note, however, that the result of Proposition \ref{P6} is non-trivial whenever one of the morphisms has a non-zero kernel. In such case, we ask whether a subring (subalgebra) of a Noetherian ring (a finitely generated $R-$algebra) that contains a non-trivial ideal of the over-ring is itself Noetherian (finitely generated as an $R-$algebra). It turns out that much more can be said in that case. 

As a result, for instance, we will be able to prove that, under some regularity conditions (if the kernels of $\varphi$ and $\psi$ both contain a non-zero-divisor), the pullback $P$ is Noetherian (finitely generated algebra over $R$) if and only if $B$ and $C$ are finitely generated modules over it (with the $P-$module structure given by the maps $\pi_1$ and $\pi_2$; see Theorem \ref{T20} for more details). Thereby, we generalize main results on pullbacks of Noetherian rings of \cite{Fon} and \cite{BW}.

\section{Subalgebras containing an ideal}\label{Sec3}
In this section, we discuss extensions of rings ($R-$algebras) $S \subseteq B$ rings such that $B$ is Noetherian (finitely generated as an algebra over $R$) and $S$ contains an ideal $I$ of $B$. Our goal is to establish when the subalgebra $S$ is also Noetherian (finitely generated as an algebra over $R$) in this case.

\begin{remark}
	Note that the setting of this section can be also phrased in terms of pullbacks. In the setting as above, the ring ($R-$algebra) $S$ with natural maps is the pullback of the following diagram:
	$$\begin{tikzcd}
		B \arrow{dr}[']{\pi_I} & & S/I \arrow{dl}{\iota}\\
		& B/I &	
	\end{tikzcd}$$ Where $\iota$ is naturally induced from the inclusion $S \subseteq B$.
\end{remark}

\begin{lemma}\label{L7}
	Let $S \subseteq B$ be an extension of rings ($R-$algebras) such that $B$ is Noetherian (finitely generated as an algebra over $R$), and suppose there exists an ideal $I$ of $B$ which lies in $S.$ If $S$ is Noetherian (finitely generated over as an algebra $R$), then $B/\mathrm{Ann}_B(f)$ is a finitely generated $S-$module for each non-zero $f \in I$.
\end{lemma}
\begin{proof} For each non-zero $f \in I$, we have $(f)_B \subseteq I \subseteq S$. As an ideal of a Noetherian ring $S$, $(f)_B$ has to be finitely generated. Consequently, there exist $b_1, \dots, b_k \in B$ such that $(f)_{B} = (fb_1, \dots, fb_k)_{S}.$ Choose an arbitrary $b \in B$; we can find $s_1, \dots, s_k \in S$ such that $fb = \sum_{i=1}^k fb_is_i.$ We deduce that $b - \sum_{i=1}^k b_is_i \in \mathrm{Ann}_B(f)$, and, thus, $b_1 + \mathrm{Ann}_B(f), \dots, b_k + \mathrm{Ann}_B(f)$ generate $B/\mathrm{Ann}_B(f)$ as an $S-$module.
\end{proof}

Now, we show that the result of the previous lemma can be strengthened considerably under additional assumptions, specifically, if the ring is coprimary as a module over itself.

\begin{definition}[Prime ideals associated to a module and coprimary module; defined on pages 89 and 94 in \cite{Eis}] Let $M$ be an $R-$module. A prime ideal $\mathfrak{p} \in \mathrm{Spec}\,R$ is associated to $M$ if $\mathfrak{p}$ is the annihilator of an element of $M.$ A submodule $N$ of $M$ is primary if only one prime is associated to $M/N$. We say that $M$ is coprimary module if its zero submodule is primary.
\end{definition}

\begin{lemma}\label{L9}
 	Suppose $B$ is a coprimary and finitely generated $R-$algebra. Then all elements of the only associated prime of $B$ are nilpotent.	
\end{lemma}
\begin{proof} The lemma follows easily from Proposition 3.9 on page 94 in \cite{Eis}.
\end{proof}

\begin{proposition}\label{P10}
 	Suppose $S \subseteq B$ is an extension of rings ($R-$algebras) such that there exists a finitely generated ideal $I \subseteq B$. If $B/I$ is finitely generated as a $B-$module and, when viewed as a ring, all its elements are nilpotent, then $B$ is a finitely generated $S-$module. 
\end{proposition}
\begin{proof} At first, we use the fact that $I$ is finitely generated and all its elements are nilpotent to find the smallest $m \in \mathbb{N}$ such that $I^m = 0$.

	Next, we proceed by induction on $m$. Let $m = 2$; otherwise, it is trivial. Let $b_1, \dots, b_n \in B$ be such that $b_i + I$ generate $B/I$ as an $S-$module and $i_1, \dots, i_\ell$ generate $I$ as an ideal in $B$. For each $b \in B$, there are $s_1(b), \dots, s_n(b)$ such that $b - \sum_i s_i(b)b_i \in I$. Also, for each $i \in I$, there are $c_1(i), \dots, c_\ell(i) \in B$ such that $i = \sum_j c_j(i) i_j$. This allows us to write $i = \sum_j \sum_k s_k(c_j) b_k i_j$ since we have that $I^2 = 0$. Thus, elements $b_i$ and $b_j i_k$ generate $B$ as a finitely generated $S$-module.

	The induction goes for $m$ as follows: we use the step $m-1$ to prove that $B/I^m$ is a finitely generated $S-$module, and we set $\tilde{I} = I/I^m$ and $\tilde{B} = B/I^m$; subsequently, we use the step $m = 2$ for $\tilde{I} = I^{m-1}$ and $\tilde{B} = B$.
\end{proof}

The results of this section up to this point can be put together in the following theorem:

\begin{theorem}\label{T11}
	Let $S \subseteq B$ be an extension of rings ($R-$algebras) such that $B$ is Noetherian (finitely generated over $R$) and coprimary as a module over itself. Moreover, suppose that there exists a non-zero ideal $I$ of $B$ which lies in $S.$ Then, $S$ is Noetherian (finitely generated over $R$) if and only if $B$ is a finitely generated $S-$module.
\end{theorem}
\begin{proof}
  $(\Leftarrow)$ This implication follows trivially from Theorem \ref{T4}.
  
  $(\Rightarrow)$ Let $\mathfrak{p} \in \mathrm{Spec}\,B$ be the only associated prime of $B$; by definition, we know that $\mathrm{Ann}_B(b) \subseteq \mathfrak{p}$ for each $b \in B$ since elements in $\mathrm{Ann}_B(b)$ are zerodivisors on $B$ (Theorem 3.1 in \cite{Eis}).
   
  Since $I$ contains a non-zero element $f$, by Lemma \ref{L7}, $B/\mathrm{Ann}_B(f)$ is a finitely generated $S-$module. It follows obviously that $B/\mathfrak{p}$ is also a finitely generated $S-$module.
  
  As $B$ is Noetherian, $\mathfrak{p}$ is a finitely generated $B-$module and by Lemma \ref{L9} all its elements are nilpotent. We can now directly apply Proposition \ref{P10} to obtain that $B$ is a finitely generated $S-$module.
\end{proof}

The theorem above can be readily applied to some special cases of pushouts:

\begin{corollary}
	Let $B$ and $T$ be Noetherian (finitely generated $R-$algebras) such that $B$ is coprimary as a module over itself. Let $I$ be a non-trivial ideal of $B$. Then, the pullback of the corresponding diagram:
	$$\begin{tikzcd}
		B \arrow{dr}[']{\pi_I} & & T \arrow{dl}{\varphi}\\
		& B/I &	
	\end{tikzcd}$$ is Noetherian (a finitely generated $R-$algebra) if and only if $B/I$ is finitely generated as a $T-$module. 
\end{corollary}
\begin{proof}
	Denote the pullback $P$ and the respective homomorphisms $\pi_B$ and $\pi_T$. We begin by observing that, since $\pi_I$ is onto, $\pi_T$ is onto. By Proposition \ref{P6}, $P$ is Noetherian (finitely generated as an $R-$algebra) if and only if so is $\pi_B(P)$.
	
	Clearly, $\pi_B(P)$ contains the ideal $I$. This allows us to apply Theorem \ref{T11}. Therefore, $\pi_B(P)$ is Noetherian (a finitely generated $R$-algebra) if and only if $B$ is a finitely generated $\pi_B(P)$ module. Since $I \subseteq \pi_B(P)$, we observe that $\pi_B(P)$ is Noetherian (a finitely generated $R-$algebra) if and only if $B/I$ is finitely generated as a $\pi_B(P)-$module. Using the fact that $\pi_I(\pi_B(P)) = \varphi(\pi_T(P)) = \varphi(T)$, we conclude the proof.
\end{proof}

It is possible to extend the result of the previous theorem to arbitrary Noetherian rings (finitely generated $R-$algebras) by using the primary decomposition of its zero ideal.

\begin{theorem}[Lasker-Noether theorem or primary decomposition; Theorem 3.10 in \cite{Eis}, page 95]\label{T12}
	Let $M$ be a finitely generated $R-$module. Any proper submodule $M'$ of $M$ is the $[\mathrm{finite}]$ intersection of primary submodules.
\end{theorem}

\begin{theorem}\label{T13}
	Assume that $S \subseteq B$ are rings ($R-$algebras) such that $B$ is Noetherian (finitely generated as an algebra over $R$). Moreover, assume there exists an ideal $I$ of $B$ which lies in $S.$ Let $P_1, \dots, P_n$ be a primary decomposition of the zero ideal in $B$ such that $I \subseteq P_i$ for all $1 \leq i \leq n'$ and that $I \nsubseteq P_j$ for every $n'+1 \leq j \leq n.$ Then, $S$ is Noetherian (finitely generated as algebra over $R$) if and only if $S/I$ is Noetherian (a finitely generated $R-$algebra) and $B/P_j$ is a finitely generated $S+P_j/P_j-$module for every $n'+1 \leq j \leq n.$
\end{theorem}
\begin{proof}
  $(\Leftarrow)$ This implication follows immediately because $S/I$ is Noetherian (finitely generated $R-$algebra) as a homomorphic image of the Noetherian (finitely generated $R-$algebra) $S$, and so are $S+P_j/P_j$ for every $n'+1 \leq j \leq n.$ Since $I \nsubseteq P_j$ for every $n'+1 \leq j \leq n$, then, $I+P_j/P_j$ is non-zero for all $j$. Therefore, using Theorem \ref{T11}, $B/P_j$ is a finitely generated $S+P_j/P_j-$module for all $j$.
  
	$(\Rightarrow)$ We know that $I \subseteq P_i$ for all $1 \leq i \leq n'$ and that $S/I$ is Noetherian (a finitely generated algebra over $R$). Thus, $S+P_i/P_i\cong S/P_i \cap I$ is also Noetherian (finitely generated $R-$algebras) for all $1 \leq i \leq n'$. As well, we have that $B/P_j$ is a finitely generated $S+P_j/P_j-$module for every $n'+1 \leq j \leq n$. By Theorem \ref{T4} and \ref{T11}, $S+P_j/P_j \cong S/P_j \cap I$ is Noetherian (a finitely generated $R-$algebra) for every $n'+1 \leq j \leq n$. By virtue of the choice of $P_1, \dots, P_n$ we have that $\bigcap_{k=1}^n P_k = \{0\}$; specifically, we get $\bigcap_{k=1}^n (I \cap P_k) = \{0\}.$ If we inductively apply Proposition \ref{P5}, we obtain that $S$ is Noetherian (a finitely generated $R-$algebra).
\end{proof}

\section{Intersections of Noetherian subrings}\label{Sec4}
In the preceding section, we dealt with the problem whether a subring (subalgebra) of a Noetherian ring (finitely generated algebra) containing its ideal is Noetherian (finitely generated as an $R$-algebra). We formulated the main result in Theorem \ref{T13}.

However, the strength of this theorem is limited. Generally, we can replace the condition of subring being Noetherian (subalgebra being finitely generated $R-$algebra) by the condition of the ring ($R$-algebra) being module-finite over the subring (subalgebra) only for some of the primary components of the algebra. In effect, it forces us to assume that some quotient of the subring (subalgebra) in question is Noetherian (finitely generated over $R$).

To illustrate this point, let us go back to the pullback discussed in the context of Proposition \ref{P6}. We have Noetherian rings (algebras finitely generated over $R$) $A,B,C$ and $\varphi: B \to A$ a $\psi: C \to A$ homomorphisms between them, and we ask whether the pullback of this diagram is Noetherian (finitely generated over $R$). 

To this end, we would like to know it about $S = \varphi^{-1}(\mathrm{Im}\,\varphi \cap \mathrm{Im}\,\psi)$. If we were to use Theorem \ref{T13}, we would encounter a problem. Typically, we need not to know anything of $S/I = \mathrm{Im}\,\varphi \cap \mathrm{Im}\,\psi$ where $I = \mathrm{Ker}\,\varphi$, which is required to be Noetherian (finitely generated as an $R$-algebra) in the Theorem \ref{T13}.

Therefore determining if the pullback of this diagram:
$$\begin{tikzcd}
	B \arrow{dr}[']{\varphi}& & C \arrow{dl}{\psi}\\
	& A &	
\end{tikzcd}$$ is Noetherian (finitely generated over $R$) using the tools of previous section, generally, requires knowing that the pullback of this diagram: 
$$\begin{tikzcd}
\mathrm{Im}\,\varphi \arrow{dr}[']{\subseteq}& & \mathrm{Im}\,\psi \arrow{dl}{\subseteq}\\
& A &	
\end{tikzcd}$$ which is isomorphic $\mathrm{Im}\,\varphi \cap \mathrm{Im}\,\psi$ is also Noetherian (finitely generated over $R$). 

In general, it is a daunting task to say whether an intersection of two Noetherian subrings of a ring is Noetherian. This is evidenced by many counterexamples listed, for example, in \cite{Bay} and \cite{Mon}. Nonetheless, under favourable circumstances, it is possible to prove that an intersection of two Noetherian subrings (finitely generated $R-$subalgberas) is Noetherian (finitely generated as an algebra over $R$).

At first, we note that the result of Theorem \ref{T13} can actually be used to establish some facts on intersections of subrings (algebras) of a ring containing an ideal of some well-behaved overring. Before getting to that result, we need to formulate some lemmas.

\begin{lemma}\label{L14}
	Assume that $S \subseteq B$ are $R-$algebras and that $I \subseteq B$ is an ideal of $B$ which is contained in $S$. Then, $B$ is a finitely generated $S-$module if and only if $B/I$ is a finitely generated $S-$module.
\end{lemma}

\begin{lemma}\label{L15}
	Suppose that $M$ is an $R-$module with submodules $M_1, \dots, M_n$ such that $M/M_i$ is Noetherian for all $1 \leq i \leq n$; then $M/M_1 \cap \dots \cap M_n$ is also Noetherian.
\end{lemma}

\begin{remark}
	The lemmas above already allow us to prove that, under some assumptions an intersection of two Noetherian rings (finitely generated $R$-algebras) is also Noetherian (finitely generated as an algebra over $R$). Let $S \subseteq B$ be rings ($R-$algebras) such that $B$ is Noetherian (finitely generated as an algebra over $R$). Also, let $I,J \subseteq B$ be ideals of $B$. If $B$ is a finitely generated as an $S+I-$module and as an $S+J-$module, $B$ is a finitely generated as an $S+(I \cap J)-$module and as an $(S+I) \cap (S+J)-$module. Furthermore, all of these rings ($R-$algebras) are Noetherian (finitely generated).
	
	This claim can be proved very easily. Since $S+(I \cap J) \subseteq (S+I) \cap (S+J),$ it suffices to prove that $B$ is a finitely generated  $S+(I \cap J)-$module if it is finitely generated as an $S+I-$module and an $S+J-$module.
	
	Clearly, $I$ and $J$ are $S+(I \cap J)-$modules. However, by Lemma \ref{L14}, we know that $B/I$ and $B/J$ are finitely generated modules over $S+I$ and $S+J$, respectively. As $S+(I \cap J)+I/I \cong S+I/I$ and as a similar statement holds for $J$, $S+(I \cap J)$ acts on $B/I$ and $B/J$ in the same way as $S+I$ and $S+J$, respectively. Hence, $B/I$ and $B/J$ are also finitely generated $S+(I \cap J)-$modules. By Lemma \ref{L15}, $B/I \cap J$ is a finitely generated $S+(I \cap J)-$module. Another usage of Lemma \ref{L15}, consequently, establishes that $B$ is a finitely generated $S+(I \cap J)-$module.
	
	All of the said rings ($R-$algebras) are obviously Noetherian (finitely generated) using Theorem \ref{T4}.
\end{remark}

In some sense, we might think of this result as of an analogue of Theorem \ref{T11}. This begs a natural question of whether it is possible to prove something corresponing to Theorem \ref{T13}. It turns out that, with the following lemma at hand, it is indeed the case.

\begin{lemma}\label{L16}
	Let $B$ be a commutative ring, and let $I,J,P$ be its ideals. Then, the following holds\: $$\sqrt{(I \cap J) + P} = \sqrt{(I+P) \cap (J + P)}.$$ 
\end{lemma}
\begin{proof}
	This can be proven using basic properties of the Zariski topology on $\mathrm{Spec}\,B$. Simply write: $V((I \cap J) + P) = V(I \cap J) \cap V(P) = (V(I) \cup V(J)) \cap V(P) = (V(I) \cap V(P)) \cup (V(J) \cap V(P)) = V(I+P) \cup V(J+P) = V((I+P) \cup (J+P)).$
\end{proof}

\begin{theorem}\label{T17}
	Let $S \subseteq B$ be rings ($R-$algebras) such that $B$ is Noetherian (finitely generated as an algebra over $R$). Also, let $I,J \subseteq B$ be ideals of $B$. If both $S+I$ and $S+J$ are Noetherian (finitely generated $R-$algebras), then $S + (I \cap J)$ and $(S+I) \cap (S+J)$ are Noetherian (finitely generated $R-$algebras).
\end{theorem}
\begin{proof}
	We know that $S+I/I$ and $S+J/J$ are Noetherian (finitely generated $R-$algebras) and that $S+(I \cap J) + I = (S+I) \cap (S+J) + I = S+I$ (a similar thing holds for $J$). As in the remark above, therefore, $S + (I \cap J) + I/I \cong S + I/I$ (and similarly for $(S + I) \cap (S + J)$ instead of $S + (I \cap J)$ and $I$ instead of $I$). By Proposition \ref{P5}, we have that $S+(I \cap J)/I \cap J$ and $(S+I) \cap (S+J)/I \cap J$ are Noetherian (finitely generated algebras over $R$).
	
	Let $P_1, \dots, P_n$ form a primary decomposition of the zero ideal in $B$ thought of as a module over itself. Without loss of generality, we can assume that for all $1 \leq i \leq n'$, we have $I \cap J \subseteq P_i$ and that for $n'+1 \leq j \leq n$, $I \cap J \nsubseteq P_J$ holds. We show that $B/P_j$ is a finitely generated $S+ (I \cap J)-$module for each $n'+1 \leq j \leq n.$ By Lemma \ref{L14}, it suffices to prove that $B/(I \cap J) + P_j$ is a finitely generated module.
	
	We know that the nilradical of $B/(I \cap J) + P_j$ is $\sqrt{(I \cap J) + P_j}/(I \cap J) + P_j.$ We recall that $(I \cap J) + P_j \subseteq (I + P_j) \cap (J + P_j)$ and $\sqrt{(I + P_j) \cap (J + P_j)} = \sqrt{(I \cap J) + P_j}$, which we proved in Lemma \ref{L16}. Therefore, $(I + P_j) \cap (J + P_j)/(I \cap J) + P_j$ is, specifically, a finitely generated nilpotent ideal of $B/(I \cap J) + P_j.$
	
	From Theorem \ref{T13}, Lemma \ref{L14}, and the fact that both $S+I$ and $S+J$ are Noetherian (finitely generated $R-$algebras), we deduce that $B/(I + P_j)$ and $B/(J+P_j)$ are finitely generated as $S-$modules (with the $S-$modules structures given by the action of $(S+I) + P_j/(I + P_j)$ and of $(S+J) + P_j/(J + P_j)$, respectively). Whereas, an application of Proposition \ref{P5} gives us that $B/(I + P_j) \cap (J + P_j)$ is a finitely generated module over $S$. Finally, this means that $B/(I \cap J) + P_j$ is a finitely generated as $S-$module by Proposition \ref{P10}.
	
	Now, we have that $B/P_j$ is a finitely generated $S+(I \cap J)-$module. As $S+(I \cap J) \subseteq (S+I) \cap (S+J)$, $B/P_j$ is \textit{a fortiori} a finitely generated as a module over $(S+I) \cap (S+J).$  Invoking Theorem \ref{T13} once more, we obtain the desired result: the rings ($R-$algebras) $S+(I \cap J)$ and $(S+I) \cap (S+J)$ are Noetherian (finitely generated).
\end{proof}

Second, we give an example of a combinatorial nature. This example is also more specific in a sense that we restrict ourselves to finitely generated algebras over fields.

\begin{proposition}\label{P18}
	Let $K$ be a field, and $S_1, S_2 \subseteq B$ be finitely generated subalgebras of a $K-$algebra $B$. Moreover, let $B$ be a domain. Then, $S_1 \cap S_2$ is a finitely generated $K-$algebra if and only if there are generators $g_1, \dots, g_m$ and $h_1, \dots, h_n$ of $S_1$ and $S_2$, respectively, such that  $S_1 \cap S_2$ as a vector space over $K$ is generated by the intersection of multiplicative sets generated by $g_1, \dots, g_m$ and $h_1, \dots, h_n$, respectively.
\end{proposition}
\begin{proof}
	$(\Rightarrow)$ Assume that $S_1 \cap S_2$ is a finitely generated $K-$algebra with generators $f_1, \dots, f_k$. Also, we assume that $g_1, \dots, g_m$ and $h_1, \dots, h_n$ are generators of $S_1$ and $S_2$, respectively. Then the claim clearly holds for the intersection of multiplicative sets generated by $1, f_1, \dots, f_k$, $g_1, \dots, g_m$ and $1, f_1, \dots, f_k, h_1, \dots, h_n$, respectively.

	$(\Leftarrow)$ Suppose that $g_1, \dots, g_m$ and $h_1, \dots, h_n$ are generators of $S_1$ and $S_2$, respectively and that $M \cap (S_1 \cap S_2)$ generates $S_1 \cap S_2$ a vector space over $K$. Here, $M$ denotes the intersection of multiplicative sets generated by $g_1, \dots, g_m$ and $h_1, \dots, h_n$, respectively. In other words, every element of $S_1 \cap S_2$ is a $K-$linear combination of elements of $M$.

	To prove that $S_1 \cap S_2$ is finitely generated, it suffices to show that there are finitely many elements of $M$ in $S_1 \cap S_2$ such that any element of $M$ in $S_1 \cap S_2$ is a product of their powers.

	The set $M$ is clearly surjective image of the following set: $$\mathfrak{I} = \{(a_1, \dots, a_m, b_1, \dots, b_n) \in \mathbb{N}_{0}^{n+m}: g_1^{a_1} \dots, g_m^{a_m} = h_1^{b_1} \dots h_n^{b_n}\}$$ under the map that sends the $(m+n)-$tuple of $\mathfrak{I}$ to the product of $g_1, \dots, g_m$ to the powers given by its first $m$ elements. 

	For the purposes of this proof, we consider $\mathbb{N}_{0}^{n+m}$ to be equipped with a natural partial ordering such that $\mathbf{a} \geq  \mathbf{b}$ if $a_i \geq b_i$ for all $1 \leq i \leq n+m.$ Suppose that $(a_1, \dots, a_m, b_1, \dots, b_n), (c_1, \dots, c_m, d_1, \dots, d_n) \in \mathfrak{I}$ such that: $$(a_1, \dots, a_m, b_1, \dots, b_n) > (c_1, \dots, c_m, d_1, \dots, d_n).$$ We observe that $(a_1-c_1, \dots, a_m-c_m, b_1-d_1, \dots, b_n-d_n) \in \mathfrak{I}$. Actually, this claim follows easily from the following equalities: $$p \cdot g_1^{a_1-c_1} \dots, g_m^{a_m-c_m} = g_1^{a_1} \dots, g_m^{a_m} = h_1^{b_1} \dots h_n^{b_n} = h_1^{b_1-d_1} \dots h_n^{b_n-d_n} \cdot p$$
where all $a_i-c_i, b_j-d_j \geq 0$ (at least one such inequality is strict) and $p = g_1^{c_1} \dots, g_m^{c_m} = h_1^{d_1} \dots h_n^{d_n}.$ Cancelling $p$ out (which is available to us by virtue of $B$ being a domain), we obtain: $$g_1^{a_1-c_1} \dots, g_m^{a_m-c_m} = h_1^{b_1-d_1} \dots h_n^{b_n-d_n}$$ what we wanted to prove.

	By Dickson's lemma, which says that every subset of $\mathbb{N}_{0}^{n+m}$ has finitely many minimal elements with respect to the natural partial order (Theorem 5 on page 71 in \cite{CLO} gives equivalent formulation in related terms of monomial ideals), there are finitely many elements $\mathbf{i}_1, \dots, \mathbf{i}_k$ in $\mathfrak{I}$ such that for each $\mathbf{a} \in \mathfrak{I}$ at least one of these elements lies beneath it. By induction on $||\mathbf{a}||_1 = \sum_{i=1}^{m+n} a_i$, we will prove that for each $\mathbf{a} \in \mathfrak{I}$ there are $c_1, \dots, c_k \in \mathbb{N}_0$ such that $\mathbf{a} = c_1\mathbf{i}_1 + \dots + c_k\mathbf{i}_k$.

	Let $\mathbf{a}$ have the minimal norm $||\mathbf{a}||_1$ over $\mathfrak{I}$; clearly, $\mathbf{a}$ is then one of $\mathbf{i}_1, \dots, \mathbf{i}_k$ as there is no element of $\mathfrak{I}$ strictly beneath it (any $\mathbf{b}$ with $\mathbf{a} > \mathbf{b}$ has to have strictly smaller norm). The induction arguments goes as follows: let $\mathbf{a} \in \mathfrak{I}$; then either $\mathbf{a}$ is one of $\mathbf{i}_1, \dots, \mathbf{i}_k$ or one of $\mathbf{i}_1, \dots, \mathbf{i}_k$ is strictly beneath $\mathbf{a}$, thus $\mathbf{a}-\mathbf{i}_j \in \mathfrak{I}$ for some $1 \leq j \leq k$. However, $||\mathbf{a}||_1 > ||\mathbf{a} -\mathbf{i}_j||_1$, which enables us to apply the inductive assumption.

	This means, by the surejection $\mathfrak{I}$ onto $M$, that each element of $M$ is a product of powers of $g_1^{i_1^1} \dots, g_m^{i_m^1}, \dots, g_1^{i_1^k} \dots, g_m^{i_m^k}$. These elements, in turn, generate $S_1 \cap S_2$ as $K-$algebra.
\end{proof}

The Proposition \ref{P18} can be readily applied to subalgebras of polynomial rings over $K$ generated by monomials.

\begin{theorem}\label{T19}
	Let $S_1, S_2 \subseteq K[x_1, \dots, x_n]$ be subalgebras generated by monomials. Then, $S_1 \cap S_2$ is finitely generated.
\end{theorem} 
\begin{proof}
	At first, we observe that for $i=1,2$ and for each $f \in S_i$ all monomials whose sum $f$ belongs to $S_i.$ However, all elements of $S_i$ are $K-$linear combinations of products of powers of its generators. Such products are indeed monomials, provided that we assume $S_i$ is generated by monomials, and clearly belong to $S_i$. Our claim then follows from the uniqueness of expression of a polynomial in $K[x_1, \dots, x_n]$ as a $K-$linear combination of monomials. Monomials form a basis of $K[x_1, \dots, x_n]$ as a vector space over $K$.

	Suppose that $f \in S_1 \cap S_2$, and $f = \sum_{i=1}^m c_if_i$ where $c_1, \dots, c_m \in K$ and $f_1, \dots, f_m$ are monomials. We know that, as $f \in S_i$, $f_1, \dots, f_m \in S_i$ for $i = 1,2$. This means that $g_1, \dots, g_m \in S_1 \cap S_2$ and that $S_1 \cap S_2$ is generated by monomials. The algebra $S_1 \cap S_2$ is generated by all its elements. However, any such element is a sum of monomials in $S_1 \cap S_2$; thus, $S_1 \cap S_2$ is generated by its monomials.

	To prove that $S_1 \cap S_2$ is finitely generated, we use the proposition above. Let $g_1, \dots, g_m$ and $h_1, \dots, h_n$ be monomial generators of $S_1$ and $S_2$, respectively. Take $M$, the intersection of multiplicative sets $M_1$ and $M_2$ generated by $1, g_1, \dots, g_m$ and $1, h_1, \dots, h_n$, respectively. Suppose that $f \in S_1 \cap S_2$ is a monomial. Then, $f \in M_1$ and $f \in M_2$ by uniqueness of expression of elements of $K[x_1, \dots, x_n]$ as sums of monomials. Consequently, $f \in M$; therefore, all monomials in $S_1 \cap S_2$ are in $M$, and $M$ generates $S_1 \cap S_2$ as a vector space over $K$.	
\end{proof}

\section{Examples and local properties}\label{Sec5}
In this section, we deal mainly with applications of the results developed above in specific situations and examples. At first, we give another partial solution to the main problem of this article---whether a pullback of a diagram of Noetherian rings (finitely generated algebras over $R$) is finitely generated. Secondly, we revisit the motivational examples given in Section \ref{Sec1}. Finally, we investigate some examples of properties of pushouts of algebraic sets arising from glueing individual points together with focus on their local properties.

We begin by proving that, in some cases, the property of being Noetherian (or a finitely generated $R$-algebra) is preserved under taking pullback. Our result generalizes the main result on pullbacks of Noetherian rings in Theorem 3.2 of \cite{BW}.
\begin{theorem}\label{T20}
	Let $A,B,C$ be Noetherian (finitely generated $R-$algebras), and let $\varphi: B \to A$ and $\psi: C \to A$ be their homomorphisms such that $\varphi^{-1}(\mathrm{Im}\,\varphi \cap \mathrm{Im}\,\psi)$ and $\psi^{-1}(\mathrm{Im}\,\varphi \cap \mathrm{Im}\,\psi)$ contain a regular ideal of $B$ and $C$, respectively. Then, the pullback of the corresponding diagram:
	$$\begin{tikzcd}
		B \arrow{dr}[']{\varphi} & & C \arrow{dl}{\psi}\\
		& A &	
	\end{tikzcd}$$ is Noetherian (finitely generated $R-$algebra) if and only if $B$ and $C$ are finitely generated as modules over $\varphi^{-1}(\mathrm{Im}\,\varphi \cap \mathrm{Im}\,\psi)$ and $\psi^{-1}(\mathrm{Im}\,\varphi \cap \mathrm{Im}\,\psi)$, respectively.
\end{theorem}
\begin{proof}
	We apply Proposition \ref{P5} and Lemma \ref{L7} to $\varphi^{-1}(\mathrm{Im}\,\varphi \cap \mathrm{Im}\,\psi) \subseteq B$ and $\psi^{-1}(\mathrm{Im}\,\varphi \cap \mathrm{Im}\,\psi) \subseteq C$. We simply set $f$ to be the non-zero-divisior in the regular ideal of $B$ contained in $\varphi^{-1}(\mathrm{Im}\,\varphi \cap \mathrm{Im}\,\psi)$ or the regular ideal of $C$ contained in $\psi^{-1}(\mathrm{Im}\,\varphi \cap \mathrm{Im}\,\psi)$, respectively. 
\end{proof}

Now, we discuss the motiovational examples of contracting a line (in the affine plane) and two points (on an affine line) to a single point with the help of the results established above:

\begin{example}[Contracting a line in $\mathbb{A}_K^2$; included without proof as Example 3.5 in \cite{Sch}, page 7]
	Consider this diagram:
	$$\begin{tikzcd}
		K[x,y] \arrow{dr}[']{\pi}& & K \arrow{dl}{\iota}\\
		& K[x,y]/(y) &
	\end{tikzcd}$$ with $\iota$ and $\pi$ being canonical inclusion and projection, respectively. The pullback of this diagram is $K+(y)$.

	Theorem \ref{T20} gives us that $K+(y)$ is finitely generated $K-$algebra if and only if $K[x,y]$ is a finitely generated $K+(y)-$module. By Lemma \ref{L14}, this is equivalent to $K[x] \cong K[x,y]/(y)$ being finitely generated $K+(y)-$module; that clearly does not hold, as, otherwise, $K[x]$ would have to be a finite dimensional vector space over $K \cong K+(y)/(y)$.

	Therefore, $K+(y)$ is not a finitely generated $K-$algebra, and it is impossible to contract a line in $\mathbb{A}_K^2$ into a point in the category of algebraic sets over $K$.
\end{example}

\begin{example}[Contracting a finite number of points on algebraic varieties over $K$; a special case given as Example 3.6 in \cite{Sch}, page 7]
	Let $L$ be a finitely generated $K-$algebra, and let $I$ be an intersection of finitely many maximal ideals of $L$. Consider this diagram:
	$$\begin{tikzcd}
		L \arrow{dr}[']{\pi}& & K \arrow{dl}{\iota}\\
		& L/I &
	\end{tikzcd}$$ with $\iota$ and $\pi$ being canonical the inclusion and projection, respectively. The pullback of this diagram is $K+I$, as in the previous example. This situation can be viewed as contracting a finite number of points to a single one.

	Suppose that $\mathfrak{m}_1, \dots, \mathfrak{m}_n$ are maximal ideals of $L$ such that $I = \bigcap_{i=1}^n \mathfrak{m}_i.$ However, by Theorem 4.19 on page 132 in \cite{Eis}, $K \subseteq L/\mathfrak{m}_i$ is a finite field extension. This provides us, by virtue of Lemma \ref{L14}, that $L$ is a finitely generated $K + \mathfrak{m}_i-$module for all $i$. Using Theorem \ref{T17} inductively, we get that $L$ is a finitely generated $K + \bigcap_{i=1}^n \mathfrak{m}_i-$module. Hence, $K+I$ is a finitely generated $K-$algebra by Theorem \ref{T4}.
\end{example}

For the rest of this section, we concern ourselves with studying glueing points on algebraic sets and thereby arising singularities. 

We prove a result whose special case shows, informally speaking, that glueing finitely many points of an algebraic variety results in a singularity of order at least of the product of dimension of the variety and the number of glued points. Also, we discuss some finer properties of $K-$algebras and singularities arising in some cases of glueing points on the affine line.

\begin{proposition}\label{P20}
	Suppose that $B$ is an integral domain finitely generated as an algebra over $R$. Suppose, furthermore, that $S \subseteq B$ is its subalgebra and $I \subseteq S$ is an ideal of $B$ which can be written as an intersection of prime ideals, $I = \bigcap_{a \in A} \mathfrak{p}_a.$ Finally, assume that $h_1 + I, \dots, h_k + I$ form a free basis of an $S/I-$module $B/I$. Provided that $\{i_1, \dots, i_n\} \subseteq I$ are linearly independent in ${\mathfrak{p}_a}_{\mathfrak{p}_a}/{\mathfrak{p}_a}_{\mathfrak{p}_a}^2$ for each $a \in A,$ then $\{i_\ell h_1, \dots, i_\ell h_k; \ell = 1, \dots, n\}$ are linearly independent in $I_I/I_I^2$. Note that, here, we think of $I_I$ as an ideal of $S_I.$
\end{proposition}
\begin{proof}
	Assume there exist coefficients $s_{\ell j} + I_I \in S_I/I_I$ with $s_{\ell j} \in S$, without loss of generality, such that for all $1 \leq \ell \leq n$ and $1 \leq j \leq k$ $\sum_{j,\ell} i_\ell h_j s_{\ell j}$ lies in $I_I^2.$ Taking $s_{\ell j} \in S$ can be justified by thinking about them as elements of the respective quotient field $Q(B)$; we can then cancel out all their denominators by multiplying with elements of $S-I$, under which is $S$ closed.

	For all $\ell$, we set $h'_\ell = \sum_{j} h_j s_{\ell j}$. Since we assume that $I \subseteq \mathfrak{p}_a$ for all $a \in A,$ it follows that $I_I^2 \subseteq {\mathfrak{p}_a}_{\mathfrak{p}_a}^2$. Hence, $i_1 h_1' + \dots + i_n h_n'$ lies in ${\mathfrak{p}_a}_{\mathfrak{p}_a}^2$ for every $a \in A$. However, we also suppose that $i_1 + {\mathfrak{p}_a}_{\mathfrak{p}_a}^2, \dots, i_n + {\mathfrak{p}_a}_{\mathfrak{p}_a}^2$ are linearly independent. Therefore, for every $1 \leq \ell \leq n$, the element $h_\ell'$ has to be in ${\mathfrak{p}_a}_{\mathfrak{p}_a}.$ Moreover, as each $h_\ell' \in B,$ $h_\ell'$ is an element of $\mathfrak{p}_a$ for all $1 \leq \ell \leq n$ and $a \in A.$

	This means that $h'_\ell \in \bigcap_{a \in A} \mathfrak{p}_a = I$ for every $1 \leq \ell \leq n.$ Take any such $\ell,$ we now know that $h'_\ell + I = 0 + I.$ Let us expand that to $\sum_{j=1}^k h_js_j + I = 0 + I.$ Since we assume that $h_1, \dots, h_\ell$ form a free basis of $B/I$ as an $S/I-$module, all $s_{\ell j} + I$ need to be zero for all possible $\ell$ and $j$. Therefore, the set $\{i_\ell h_1, \dots, i_\ell h_k; \ell = 1, \dots, n\}$ is linearly independent in $I_I/I_I^2.$
\end{proof}

\begin{remark}
	The proposition above can conveniently applied in the case where $B$ is a coordinate ring of a $K-$algebraic variety $X$. Suppose that the ground field $K$ is algebraically closed and that $S = K + I$. We set $I$ to be the ideal of $B$ such that $I = \bigcap_{i=1}^n \mathfrak{m}_i$ for some maximal ideals $\mathfrak{m}_1, \dots, \mathfrak{m}_n$ of $B.$ Then $K+I$ can be thought of as a coordinate ring of an algebraic variety $X$ with finitely many points corresponding to the ideals $\mathfrak{m}_1, \dots, \mathfrak{m}_n$ are identified. See the latter example above.
	
	We can observe that $B/I$ is a finite dimensional vector space as in the example; hence, it is a free module over $S/I \cong K$. Also, $I$ is maximal, thus prime, ideal of $K+I.$ Let us now suppose that $\mathfrak{m}_i = (x_1-a_{1i}, \dots, x_m - a_{mi})/I(X)$ for all $1 \leq i \leq n$ with all $a_{ji} \in K$. For simplicity, let us also assume that $a_{ji} \neq a_{j'i}$ for $j \neq j'$. This guarantees that $x_j - a_{ji} \notin \mathfrak{m}_{i'}/I(X)$ for $i' \neq i.$
	
	It is clear that $(x_1-a_{1i}) + \mathfrak{m}_{i\,\mathfrak{m}_i}^2, \dots, (x_m - a_{mi}) + \mathfrak{m}_{i\,\mathfrak{m}_i}^2$ generate $\mathfrak{m}_{i\,\mathfrak{m}_i}/\mathfrak{m}_{i\,\mathfrak{m}_i}^2$ for each $1 \leq i \leq n.$ Thus, there is a set of indices $J_i$ such that $(x_1-a_{j_1^i i}) + \mathfrak{m}_{i\,\mathfrak{m}_i}^2, \dots, (x_{j_i^i} - a_{j_i^i i}) + \mathfrak{m}_{i\,\mathfrak{m}_i}^2$ form a basis of $\mathfrak{m}_{i\,\mathfrak{m}_i}/\mathfrak{m}_{i\,\mathfrak{m}_i}^2.$ Without loss of generality, we assume that $|J_1| = k$ is the smallest of such indices. We observe that $(x_{j_1^1}-a_{j_1^1 1})(x_{j_1^2}-a_{j_1^2 2}) \dots (x_{j_1^n}-a_{j_1^n n}) + I(X), \dots, (x_{j_k^1}-a_{j_k^1 1})(x_{j_k^2}-a_{j_k^2 2}) \dots (x_{j_k^n}-a_{j_k^n n}) + I(X)$ belong to $\mathfrak{m}_1 \dots \mathfrak{m}_n/I(X) \subseteq I.$
	
	Now, notice that these $k$ elements of $\mathfrak{m}_1 \dots \mathfrak{m}_n$ are linearly independent in $\mathfrak{m}_{i\,\mathfrak{m}_i}/\mathfrak{m}_{i\,\mathfrak{m}_i}^2$, for each $1 \leq i \leq n$, since $(x_{j_\ell^1}-a_{j_\ell^1 1})(x_{j_\ell^2}-a_{j_\ell^2 2}) \dots (x_{j_\ell^n}-a_{j_\ell^n n}) + \mathfrak{m}_{i\,\mathfrak{m}_i}^2$ and $(x_{j_\ell^i}-a_{j_\ell^i i}) + \mathfrak{m}_{i\,\mathfrak{m}_i}^2$ span the same subspace of $\mathfrak{m}_{i\,\mathfrak{m}_i}/\mathfrak{m}_{i\,\mathfrak{m}_i}^2$. This is due to the fact that for all $i$ and $j$ we have that $x_j - a_{ji} \notin \mathfrak{m}_{i'}/I(X)$ for $i' \neq i.$
	
	By the proposition above, the $S_I/I_I-$dimension of $I_I/I_I^2$ in $S_I$ is at least $k\cdot\dim_K B/I$; so the smallest dimension of a tangent space among identified points $(a_{11}, \dots, a_{1m}), \dots, (a_{n1}, \dots, a_{nm})$ times the number of points which is equal to $\dim_K B/I.$ Therefore, the order of the arising singularity is at least proportional to the number of identified points.
\end{remark}

The singularities arising from glueing finitely many points on affine algebraic sets can be studied even more closely by looking at corresponding rings of formal power series. In treating the following example, we assume that $\mathrm{char}\,K = 0$.

\begin{example}[Glueing finitely many points on $\mathbb{A}_K^1$]
	Let $a_1, \dots, a_n \in K$ be distinct. Then the ideal $I$ defining $\{a_1, \dots, a_n\}$ is generated by $\varphi_0(x) = (x-a_1) \dots (x-a_n)$. Our goal now is to describe $K+I$.
	
	We know that $K+I$ is a finitely generated algebra over $K.$ At first, we find its generators and relations between them. We claim that $K+I$ is generated by $\varphi_0(x), \dots, \varphi_{n-1}(x)$ where $\varphi_{i+1}(x) = x \varphi_i(x)$ for $0 \leq i \leq n-2$. The proof this claim goes by induction on degree of the non-zero polynomial $f \in K+I$. Actually, it suffices to assume that $f \in I.$

	As $a_1, \dots, a_n \in K$ are distinct, there are no polynomials of degree less than $n$ in $I$, and there is, up to a multiple by an element of $K$, only one polynomial of degree $n$, $\varphi_0(x).$

	Let $f \in I$ be of degree $m > n$. By the inductive assumption, we know that all polynomials of smaller degree belong to $K[\varphi_0(x), \dots, \varphi_{n-1}(x)]$. Write $m = kn + r,$ where $1 \leq k$ and $0 \leq r \leq n-1$. Denote $\ell$ the leading coefficient of $f.$ Then, $f-\ell \varphi_0^{k-1}\varphi_{r}$ is of strictly smaller degree. We can conclude the proof by pointing out that both $f-\ell \varphi_0^{k-1}\varphi_{r}$ and $\ell \varphi_0^{k-1}\varphi_{r}$ belong to $K[\varphi_0(x), \dots, \varphi_{n-1}(x)]$.

	We know that $K+I \cong K[x_0, \dots, x_{n-1}]/J$ with $J$ an ideal of $K[x_0, \dots, x_{n-1}]$, which is given by the homorphism extending $x_i \mapsto \varphi_i$ for $0 \leq i \leq n-1.$ We claim that $J$ is generated by two types of relations: $x_{i}x_{j} = x_{k}x_{\ell}$ for all $i+j = k + \ell$ and $x_{i}x_{j}x_{k} = b_n x_{n-1} x_{i+j+k+1} + \sum_{\ell=1}^{n} b_{n-\ell} x_{n - \ell} x_{i + j + k}$ where $i+j+k \leq n-2$ and $b_{n-\ell}$ are coefficients of $\varphi_0$.

	The proof is not difficult and the strategy is to proceed by induction on degree of a relation as follows:
	\begin{enumerate}
		\item Show that if $i_1 + \dots + i_m = j_1 + \dots + j_m$, then $x_{i_1} \dots x_{i_m} - x_{j_1} \dots x_{j_m} \in J$.
		\item Let $F(x_0, \dots, x_{n-1}) \in J$ be a homogenous relation of degree $m$. Observe that it can be rewritten as a sum of homogenous relations of the same degree $F(x_0, \dots, x_{n-1}) = \sum_{i=0}^{m(n-1)} F_i(x_0, \dots, x_{n-1})$. The polynomials $F_i(x_0, \dots, x_{n-1})$ are a $K$-linear combinations of monomials $x_{j_1} \dots x_{j_m}$ where $i = j_1 + \dots + j_m$ and whose coefficients add up to zero. Inductively, using the claim  above, it is possible to deduce that all $F_i(x_0, \dots, x_{n-1}) \in J$ (as, under $x_i \mapsto \varphi_i$, $F_i$ is mapped to a homogenous polynomial in $x$ of degree $i$) and so $F(x_0, \dots, x_{n-1}) \in J.$
		\item Suppose that $F(x_0, \dots, x_{n-1}) \in J$ is a general relation of degree $m$. Denote $F_m(x_0, \dots, x_{n-1})$ the $m-$th homogenous part of $F$. For a part $F_m'$ of $F_m$, we can lower the degree using the relations of the second type. On the other hand, we show that $F_m - F_m'$ is mapped (under the extension of $x_i \mapsto \varphi_i$) to polynomials in $x$ of degree high enough such that we need to have $F_m - F_m' \in J$. This gives us that $F - F_m + F_m' \in J$; however, lowering the degree of $F_m'$ results in $F - F_m + F_m'$ being of degree $m-1$, at most, allowing us to proceed by induction on degree.
	\end{enumerate}

	Finally, let us examine two specific cases: glueing two and three points on $\mathbb{A}_K^1$. Recall that we work under the assumption that $\mathrm{char}\,K = 0$ and that $K$ is algebraically closed. For convenience, we glue roots of unity in both cases.

	Identifying roots of $x^2-1$ on $\mathbb{A}_K^1$, we get a variety $V_2$ with the following coordinate ring: $$K[x_0,x_1]/(x_0^3 - x_1^2 + x_0^2)$$ with $x_0^3 - x_1^2 + x_0^2$ being an instance of the rule of the second type. To examine the resulting singularity at $0$ closely, we move to the ring of formal power series of this variety at $0$. We get $K[[x_0,x_1]]/(x_0^2(1+x_0)-x_1^2)$ (consult chapter 7 of \cite{Eis} for details). However, we can take $u \in K[[x_0,x_1]]$ as a formal square root of $1+x_0$, which is also invertible. The ring $K[[x_0,x_1]]/(x_0^3 - x_1^2 + x_0^2)$ is, thus, isomorphic to: $$K[[y_0,y_1]]/((y_0-y_1)(y_0+y_1)).$$ This means that the resulting singularity looks locally like a pair of intersecting lines.

	Identifying roots of $x^3-1$ on $\mathbb{A}_K^1$, by using our results above, we get a variety $V_3$ with the following coordinate ring: $$K[x_0,x_1,x_2]/(x_0x_2 - x_1^2, x_0^3 - x_1x_2 + x_0^2, x_0^2x_1-x^2_2+x_0x_1).$$ The first relation is of the first type; the latter two are of the second type. As above, we examine the corresponding ring of formal power series. Rewrite the latter two relations as $x_0^2(1+x_0)-x_1x_2$ and $x_0x_1(1+x_0)-x_2^2.$ We are able to find $v \in K[[x_0, x_1, x_2]]$ such that $v^3 = 1+x_0$. Furthermore, this $v$ is invertible. Rewrite the relations as $x_0 (x_2 v^{-2})-(x_1 v^{-1})^2,$ $x_0^2 - (x_1 v^{-1})(x_2 v^{-2}),$ and $x_0(x_1 v^{-1}) - (x_2 v^{-2})^2.$ This means that $K[[x_0,x_1,x_2]]/(x_0x_2 - x_1^2, x_0^3 - x_1x_2 + x_0^2, x_0^2x_1-x^2_2+x_0x_1)$ is isomorphic to: $$K[[y_0,y_1,y_2]]/(y_0^2-y_1y_2, y_1^2-y_0y_2, y_2^2-y_0y_1).$$
	
	Take $K[y_0,y_1,y_2]/(y_0^2-y_1y_2, y_1^2-y_0y_2, y_2^2-y_0y_1)$, and set $y_1 = a$ for non zero $a \in K.$ This results in $a^2 = y_1y_2,$ $y_1^2 = ay_2$, and $y_2^2 = ay_1.$ Take $y_2 = \frac{a^2}{y_1}$; then both remaining equations can be written as $y_1^3 = a^3.$ Denote $\xi_1, \xi_2, \xi_3$ three distinct roots of $x^3-1$ in $K$, where $\xi_1 = 1$ and $\xi_1, \xi_2$ are roots of $x^2-x+1$. We have three solutions $(a,a,a)$, $(a,\xi_1 a, \xi_2 a)$, and $(a,\xi_2 a, \xi_1 a).$ Choosing $y_1 = 0,$ we get $y_1 = y_2 = 0.$ This means that $K[y_0,y_1,y_2]/(y_0^2-y_1y_2, y_1^2-y_0y_2, y_2^2-y_0y_1)$ is the coordinate ring of three lines which span $K^3$ as a vector space (vectors $(1,1,1), (1,\xi_1, \xi_2),$ and $(1,\xi_2,\xi_1)$ are linearly independent by regularity of Vandermonde matrix, since $\xi_1^2 = \xi_2$ and $\xi_2^2 = \xi_1$).

	The singularity of $V_3$ at $0$, hence, looks locally like three distinct lines intersecting in a single point.
\end{example}

\section*{Acknowledgments}
This article is a substantially revised and extended version of a part of the author's master thesis defended at the Charles University, Faculty of Mathematics and Physics in September 2018. The author would like to extend his thanks to the thesis supervisor Jan \v{S}\v{t}ov\'{i}\v{c}ek for all his invaluable help in the process of writing the thesis and this text.

The author gratefully acknowledges support from grant \textit{GA\v{C}R 17-23112S} of the Czech Science Foundation.

\bibliographystyle{plain}
\bibliography{bibliography}

~\\[0.05cm]
\noindent
\textsc{Charles University, Faculty of Mathematics and Physics, Department of Algebra, Sokolovska 83, 186 75 Praha, Czech Republic}\\\\
\textit{Email address:} \texttt{jakub.kopriva@outlook.com}

\end{document}